\documentclass[12pt]{article}
\usepackage{amssymb}
\usepackage{amsmath}
\usepackage{amsthm}
\usepackage{yhmath}
\usepackage{mathdots}
\usepackage{MnSymbol}
\usepackage{a4wide}
\usepackage{color}

\usepackage[left=2.5cm,top=2cm,right=2cm]{geometry}
\usepackage{enumitem}
\usepackage{nccmath}
\usepackage{blkarray}
\usepackage{multirow}
\usepackage{float}
\usepackage{booktabs}
\usepackage{array}
\usepackage[square,numbers]{natbib} 
\usepackage{graphicx}
\graphicspath{ {/Users/neha/Desktop/NLA P2/c2_CMA} }
\usepackage{subfigure}

\usepackage{natbib}
\usepackage[english]{babel}
\usepackage{algorithm}
\usepackage{algpseudocode}

\newtheorem{theorem}{\bf Theorem}[section]

\newtheorem{exam}[theorem]{\bf Example}

\newtheorem{lemma}[theorem]{\bf Lemma}

\SetLabelAlign{CenterWithParen}{\hfil(\makebox[1.0em]{#1})\hfil}
\def \R{{\mathbb R}}
\def \C{{\mathbb C}}

\def \QR{\mathbb{Q}_\mathbb{R}}
\def \i{\textit{\textbf{i}}}
\def \j{\textit{\textbf{j}}}
\def \k{\textit{\textbf{k}}}
\newcommand\norm[1]{\left\lVert#1\right\rVert}
\newcommand{\Rn}[1]{%
	\textup{\lowercase\expandafter{\romannumeral#1}}%
}

\newcommand{\beano}{\begin{eqnarray*}}	\newcommand{\eeano}{\end{eqnarray*}}
\renewcommand{\thefootnote}{\fnsymbol{footnote}}

\date{\today}
\title{On the specific solutions of reduced biquaternion equality constrained least squares problem and their  relative forward error bound
	
	\footnotemark[2]}
\author{Sk. Safique Ahmad\footnotemark[1] \and  Neha Bhadala\footnotemark[1]}

\begin{document}
	\maketitle
	\begin{abstract}
		This study focuses on addressing the challenge of solving the reduced biquaternion equality constrained least squares (RBLSE) problem. We develop algebraic techniques to derive real and complex solutions for the RBLSE problem by utilizing the real and complex forms of reduced biquaternion matrices. Furthermore, we propose algorithms and provide a detailed analysis of their computational complexity for finding special solutions to the RBLSE problem. A perturbation analysis is conducted, establishing an upper bound for the relative forward error of these solutions. This analysis ensures the accuracy and stability of the solutions in the presence of data perturbations, which is crucial for practical applications where errors arising from input inaccuracies can cause deviations between computed and true solutions. Numerical examples are presented to validate the proposed algorithms, demonstrate their effectiveness, and verify the accuracy of the established upper bound for the relative forward errors. These findings lay the groundwork for exploring applications in 3D and 4D algebra such as robotics, signal, and image processing, expanding their impact on practical and emerging domains.
		
	\end{abstract}
	
	
	
	\renewcommand{\thefootnote}{\fnsymbol{footnote}}
	
	\footnotetext[1]{
		Department of Mathematics, Indian Institute of Technology Indore, 453535, Madhya Pradesh, India. \texttt{Contact:safique@iiti.ac.in, bhadalaneha@gmail.com}.}

	\section{Introduction}\label{sec1}
	Reduced biquaternions, first introduced by Segre in $1892$, have become fundamental in various domains such as signal processing, image processing, and neural networks \cite{el2022linear, gai2023theory, gai2021reduced, guo2011reduced}. Unlike standard quaternions, reduced biquaternions, often referred to as commutative quaternions, possess the unique property of commutative multiplication. This property not only simplifies mathematical operations but also significantly boosts computational efficiency, as demonstrated in \cite{pei2004commutative, MR2517402}.
	
	Reduced biquaternions also exhibit significant differences in their algebraic structure compared to traditional quaternions. Two notable elements, \( e_1 \) and \( e_2 \), are defined as:
		\[
		e_1 = \frac{1 + \j}{2}, \quad e_2 = \frac{1 - \j}{2}.
		\]
		These elements satisfy \( e_1e_2 = 0 \), \( e_1^2 = e_1 \), \( e_2^2 = e_2 \), and are idempotent as well as divisors of zero. Any reduced biquaternion element of the form $c_1e_1$ or $c_2e_2$ (where \( c_1 \) and \( c_2 \) are complex numbers) also acts as a divisor of zero, highlighting the non-division algebra property of reduced biquaternions. Consequently, \( \QR \), the reduced biquaternion algebra, does not form a complete division algebra. 
		
		The concept of units in \( \QR \) is also unique. An element \( r \in \QR \) is defined as a unit if there exists another element \( s \in \QR \) such that \( rs = sr = 1 \). However, not every nonzero element in \( \QR \) is a unit, unlike in \( \R \), \( \C \) or in quaternions. For instance, $e_1$ or $e_2$ is not a unit in $\QR$. These differences underline the distinct algebraic framework of reduced biquaternions and motivate their applicability to specific problems in mathematical and physical sciences.
		
		While reduced biquaternions have several advantages, their non-division algebra nature has negligible impact on practical applications. For instance, in \cite{gai2023theory, guo2011reduced, pei2004commutative, MR2517402}, the authors demonstrated the superior performance of reduced biquaternions over traditional quaternions in signal and image processing. This makes reduced biquaternions particularly advantageous for real-time applications in these domains. Moreover, their applicability extends beyond numerical computations to areas such as theoretical physics, including their connection to Maxwell’s equations \cite{MR2377520, guo2022algebraic}.
	
	The wide-ranging applications of reduced biquaternions have spurred significant research into both the theoretical and numerical aspects of matrices and matrix equations involving commutative quaternions. For instance, algorithms for calculating eigenvalues, eigenvectors, and singular value decompositions of reduced biquaternion matrices are detailed in \cite{MR2517402}. Zhang et al. explored the singular value decomposition and generalized inverse of these matrices in \cite{zhang2024singular}, and further investigated the diagonalization process in \cite{MR4685147}, where they established the necessary and sufficient conditions for diagonalization and introduced two numerical methods to facilitate this task. In \cite{MR4730574}, authors discussed the LU decomposition of reduced biquaternion matrices. In \cite{MR4649293}, Guo et al. explored eigen-problems of reduced biquaternion matrices, while in \cite{MR4137050}, Yuan et al. discussed the Hermitian solution of the reduced biquaternion matrix equation $\left(AXB, CXD\right) = \left(E, F\right)$. In \cite{MR4509835}, Ding et al. explored the solutions to matrix equations $XF-AX=BY$ and $XF-A\widetilde{X}=BY$ over commutative quaternions.
	
	In many practical applications, determining the solution to linear system, typically expressed as $AX \approx B$, is a common challenge. The least squares method is a well-known approach for addressing this problem. Prior research, such as the work in \cite{MR4085494}, has explored least squares solutions for matrix equations like $AX = B$ and $AXC = B$ over commutative quaternions. The equality constrained least squares problem is a further study of the least squares problem, and in \cite{MR4532609}, the authors discussed  solution techniques for computing the reduced biquaternion solution of the reduced biquaternion equality constrained least squares (RBLSE) problem. Recent work has extended these studies to quaternions and split quaternions, proposing algorithms for solving split quaternion and quaternion equality constrained least squares problems \cite{MR4619517, MR4509107}.
	
	Despite these advancements, existing studies primarily focus on computational techniques for finding general solutions to LSE problems over quaternions, split quaternions, and reduced biquaternions. However, specialized solutions to the RBLSE problem—such as real and complex solutions—have not been explored. Furthermore, the perturbation analysis of these solutions have not been addressed.
		
		In numerical analysis, the concept of relative forward error is central to evaluating the accuracy and stability of computed solutions, particularly for the RBLSE problem. This metric is crucial in the presence of data perturbations, where errors from machine precision, floating-point arithmetic, or input inaccuracies can lead to deviations between computed and true solutions. The relative forward error quantifies these deviations, providing insights into the sensitivity of solutions to data perturbations and enabling the assessment of solution accuracy and reliability. Understanding these errors is essential in practical applications, such as robotics, image processing, and control systems, where reduced biquaternion algebra plays a key role.
		
		To address these gaps, this study aims to:
		\begin{enumerate}
			\item Develop algebraic techniques to obtain real and complex solutions of the RBLSE problem.
			\item Establish upper bounds for the relative forward error associated with these solutions to ensure accuracy and stability.
			\item Propose efficient algorithms for solving the RBLSE problem and analyze their computational complexity.
			\item Conduct numerical experiments to validate the proposed methods, assess their computational efficiency, and verify the accuracy of the established upper bounds.
		\end{enumerate}
		Among the proposed solutions, the computational complexity analysis reveals that finding the real solution to the RBLSE problem requires less computational time than obtaining the complex solution. Numerical experiments validate the computational efficiency of the algorithms, while the established upper bounds for the relative forward error provide insights into solution sensitivity and reliability under perturbations.
		
		These contributions are vital for advancing the numerical analysis of reduced biquaternions and ensuring the robustness of solutions in applications such as signal processing, color image processing, and robotics. By addressing both computational techniques and perturbation analysis, this study lays the groundwork for further exploration of reduced biquaternion algebra in practical and emerging domains.
	
	The organization of the manuscript is as follows: Section \ref{sec2} covers the notation and preliminary concepts. Section \ref{sec4} outlines the method for finding the real solution of the RBLSE problem, while Section \ref{sec3} focuses on the technique for finding complex solutions to the same problem. Lastly, Section \ref{sec5} presents the numerical validation of the proposed methods.
	
	\section{Notation and preliminaries}\label{sec2}
	\subsection{Notation}
	In this paper, we adopt the following notations: $\QR^{m \times n}$, $\C^{m \times n}$, and $\R^{m \times n}$ represent the sets of all $m \times n$ reduced biquaternion, complex, and real matrices. The notation $\norm{\cdot}_F$ represents the Frobenius norm, which applies to matrices over real numbers, complex numbers, or reduced biquaternions, while $\norm{\cdot}_2$ denotes the 
		$2$-norm, specifically for real or complex matrices.  For matrices $A \in \QR^{m \times n}$ and $B \in \QR^{m \times d}$, $\left[A, B\right]$ represents their concatenation, i.e., $\begin{bmatrix}A & B\end{bmatrix} \in \QR^{m \times (n+d)}$. For a matrix $A \in \C^{m \times n}$, the symbols $A^T$, $A^H$, and $A^{\dagger}$ correspond to the transpose, conjugate transpose, and Moore-Penrose inverse of $A$, respectively. Additionally, $A^{-1}$ indicates the inverse of $A$ if it exists.
	
	The Matlab function $rand(m, n)$ generates an $m \times n$ matrix filled with uniformly distributed random values. The abbreviations RBLSE (reduced biquaternion equality constrained least squares) and LSE (equality constrained least squares) will be used throughout this paper.
	\subsection{Preliminaries}
	A reduced biquaternion is defined as  $\zeta= p_{0}+ p_{1}\i+ p_{2}\j+ p_{3}\k$, which can also be expressed as $\zeta= \left(p_0 + p_1 \i\right)+\left(p_2 + p_3\i\right)\j = r_1 +r_2\j$. Here,
	$p_{i} \in \R$ for $i= 0, 1, 2, 3$ and $r_{i} \in \C$ for $i= 1, 2$. The defining properties include $\i^2=  \k^2= -1$ and $\j^2= 1$, with multiplication rules
	$\i\j= \j\i= \k, \; \j\k= \k\j= \i,$ and $\k\i= \i\k= -\j$. The norm of $\zeta$ is computed as $\norm{\zeta}= \sqrt{p_0^2+ p_1^2+ p_2^2+ p_3^2} = \sqrt{|r_1|^2+|r_2|^2}$. For a matrix $M= \left(m_{ij}\right) \in \QR^{m \times n}$, the Frobenius norm \cite{MR4137050} is given by: 
	\begin{equation}\label{eq2.1}
		\norm{M}_F= \sqrt{\sum_{i= 1}^{m} \sum_{j= 1}^{n} \norm{m_{ij}}^{2}}.
	\end{equation}
	Suppose $M=M_0+ M_1\i+ M_2\j+ M_3\k= N_1+N_2\j \in \QR^{m \times n}$, where $M_i \in \R^{m \times n}$ for $i=0,1,2,3$ and $N_i \in \C^{m \times n}$ for $i=1,2$. The real and complex representations of matrix $M$, denoted as $M^R$ and $M^C$, are given by:
	\begin{equation}\label{eq2.2}
		M^R= \begin{bmatrix}
			M_0  &  -M_1  &  M_2  &  -M_3 \\
			M_1   &   M_0  &  M_3  &   M_2  \\
			M_2   &  -M_3 & M_0  &  -M_1  \\
			M_3  &  M_2  &  M_1  &    M_0 
		\end{bmatrix}, \quad M^C= \begin{bmatrix}
			N_1 & N_2 \\
			N_2 & N_1 
		\end{bmatrix}.
	\end{equation}
	Let $M_c^R$ represent the first block column of the matrix $M^R$. More precisely, $M_c^R$ is defined as: $M^R_c= 
	\left[M_0^T, M_1^T, M_2^T, M_3^T\right]^T.$ The matrix $M^R$ can be represented using $M^R_c$ as:
	\begin{equation}\label{eq2.3}
		M^R= \left[M_c^R, Q_m M_c^R, R_m M_c^R, S_m M_c^R\right], 
	\end{equation}
	where
	\begin{equation}\label{Qm}
		Q_m= \begin{bmatrix}
			0      &   -I_m   &  0     & 0    \\
			I_m  &    0        &  0     & 0    \\
			0      &    0        &  0     & -I_m \\
			0      &    0        &  I_m & 0
		\end{bmatrix}, \quad
		R_m= \begin{bmatrix}
			0      &   0      &  I_m   & 0           \\
			0      &    0     &  0       &   I_m     \\
			I_m  &    0     &  0       & 0           \\
			0      &    I_m &  0       & 0
		\end{bmatrix}, \quad
		S_m= \begin{bmatrix}
			0      &   0      &  0       & -I_m           \\
			0      &    0     &  I_m    &  0    \\
			0      &  -I_m  &  0       & 0           \\
			I_m   &    0     &  0       & 0
		\end{bmatrix}.
	\end{equation}
	The Frobenius norm, as described in \cite{MR4085494, MR4532609}, is defined using the real representation of reduced biquaternion matrices. This definition aligns with the definition provided in \eqref{eq2.1}, ensuring consistency. Specifically, the Frobenius norm of \( M \) satisfies:
	\begin{equation}\label{ec}
		\norm{M}_F= \frac{1}{2} \norm{M^R}_F =  \norm{M^R_c}_F.
	\end{equation}
	\begin{lemma}\label{lem2}
		Let $P \in \R^{m \times n}$, $Q \in \R^{m \times d}$, $R \in \R^{m \times p}$, and $S \in \R^{m \times q}$. If $\|P\|_F=\|Q\|_F=\|R\|_F=\|S\|_F$, then we have
		\begin{equation*}
			\norm{P}_F= \frac{1}{2}\norm{\left[P, Q, R, S\right]}_F.
		\end{equation*}
	\end{lemma}
	Let $M_c^C$ represent the first block column of the matrix $M^C$. More precisely, $M_c^C$ is defined as: $M^C_c= 
	\left[N_1^T, N_2^T\right]^T.$ The matrix $M^C$ can be represented using $M^C_c$ as: 
	\begin{equation}\label{eq2.4}
		M^C= \left[M_c^C, P_m M_c^C\right],  
	\end{equation}
	where 
	\begin{equation}\label{Pm}
		P_m= \begin{bmatrix}
			0 & I_m \\
			I_m & 0
		\end{bmatrix}.
	\end{equation}
	The Frobenius norm, as described in \cite{MR4085494, MR4532609}, is also expressed using the complex representation of reduced biquaternion matrices. This definition is consistent with the previously stated definition \eqref{eq2.1}. Specifically, the Frobenius norm of \( M \) satisfies:
	\begin{equation}\label{er}
		\norm{M}_F= \frac{1}{\sqrt{2}} \norm{M^C}_F =  \norm{M^C_c}_F.
	\end{equation}
	\begin{lemma}\label{lem22}
		Let $P \in \C^{m \times n}$ and $Q \in \C^{m \times d}$. If $\|P\|_F=\|Q\|_F$, then we have
		\begin{equation*}
			\norm{P}_F= \frac{1}{\sqrt{2}}\norm{\left[P, Q\right]}_F.
		\end{equation*}
	\end{lemma}
	\begin{lemma}\label{lem2.3}
		For $\alpha \in \R$, $\beta \in \C$, $P, Q \in \QR^{m \times n}$, and $R \in \QR^{n \times t}$, the following relationships hold:
		\begin{enumerate}
			\item[$(a)$]  $P= Q \iff P^C= Q^C \iff P^R=Q^R$. 
			\item[$(b)$] $\left(P+Q\right)^R= P^R+ Q^R$, $\left(P+Q\right)^C= P^C+ Q^C$.
			\item[$(c)$] $\left(\alpha P\right)^R=  \alpha P^R$, $\left(\beta P\right)^C= \beta P^C$.
			\item[$(d)$] $\left(PR\right)^R = P^R R^R$, $\left(PR\right)^C = P^C R^C$.
		\end{enumerate}
	\end{lemma}
	
	In the context of solving constrained least squares problems, the following result provides the closed-form solution for the complex LSE problem. This theorem, based on the QR factorization technique, is foundational for deriving efficient algorithms for more generalized problems discussed later in this paper. The result is formally stated as follows:
		\begin{lemma}{\rm \cite{MR3024913}} \label{thm2.4}
			Consider the complex LSE problem:
			\begin{equation}\label{eq2.4.1}
				\min_X \|A X - B\|_F \quad \text{subject to} \quad C X = D,
			\end{equation}
			where \( A \in \mathbb{C}^{m \times n} \), \( B \in \mathbb{C}^{m \times d} \), \( C \in \mathbb{C}^{p \times n} \), and \( D \in \mathbb{C}^{p \times d} \). Assume that the matrix \( C \) has full row rank. Let the QR factorization of \( C^H \) be given by:
			\begin{equation}\label{eq2.4.2}
				C^H = Q \begin{bmatrix}
					R \\ 
					0
				\end{bmatrix},
			\end{equation}
			where \( Q \in \mathbb{C}^{n \times n} \) is a unitary matrix and \( R \in \mathbb{C}^{p \times p} \) is a nonsingular upper triangular matrix. Further, partition the product \( A Q \) as:
			\begin{equation}\label{eq2.4.3}
				A Q = \left[ P_1, P_2\right],
			\end{equation}
			where \( P_1 \in \mathbb{C}^{m \times p} \) and \( P_2 \in \mathbb{C}^{m \times (n - p)} \).
			
			The unique solution \( X \in \mathbb{C}^{n \times d} \) with minimum norm for the complex LSE problem \eqref{eq2.4.1} is given by:
			\begin{equation}\label{eq2.4.4}
				X = Q \begin{bmatrix}
					\left(R^H\right)^{-1} D \\ 
					P_2^\dagger \left(B - P_1 \left(R^H\right)^{-1} D\right)
				\end{bmatrix}.
			\end{equation}
		\end{lemma}
		\begin{proof}
			Equation \eqref{eq2.4.1} can be rewritten as
			\begin{equation}\label{eq2.4.5}
				\min_{X}\norm{AQQ^HX-B}_F \quad \textrm{subject to} \quad  CQQ^HX=D.
			\end{equation}
			Set
			\begin{equation*}
				Q^H X = \begin{bmatrix}
					Y \\ 
					Z
				\end{bmatrix},
			\end{equation*}
			where \( Y \in \mathbb{C}^{p \times d} \) and \( Z \in \mathbb{C}^{(n - p) \times d} \). Utilizing \eqref{eq2.4.2}, we have
			\begin{equation*}
				CQQ^HX=D \implies \begin{bmatrix}
					R^H & 0
				\end{bmatrix}\begin{bmatrix}
					Y \\
					Z
				\end{bmatrix}=D.
			\end{equation*} Since $R^H$ is a nonsingular matrix, we get $Y=\left(R^H\right)^{-1}D$. Using \eqref{eq2.4.3}, equation \eqref{eq2.4.5} takes the form
			\begin{equation*}
				\min_{Z }\norm{P_2Z-\left(B-P_1Y\right)}_F .
			\end{equation*} The minimum norm solution of the above least squares problem is  $Z=P_2^{\dagger}\left(B-P_1Y\right)$. 
			
			Finally, substituting \( Y \) and \( Z \) back into \( Q^H X = \begin{bmatrix} Y \\ Z \end{bmatrix} \), we can derive the desired expression for $X$.
			This completes the proof.
	\end{proof}
	
	\section{An algebraic technique for real solution of RBLSE problem}\label{sec4}
	This section focuses on an algebraic approach to derive the real solution for the RBLSE problem. The method is based on analyzing the solution of the associated real LSE problem. Suppose
	\begin{eqnarray}
		A&=& A_0+ A_1\i+ A_2\j+ A_3\k \in  \QR^{m \times n}, \quad B= B_0+ B_1\i+ B_2\j+ B_3\k \in \QR^{m \times d}, \label{eq3.1} \\ 
		C&=& C_0+ C_1\i+ C_2\j+ C_3\k \in  \QR^{p \times n}, \quad D= D_0+ D_1\i+ D_2\j+ D_3\k \in \QR^{p \times d}. \label{eq3.2}
	\end{eqnarray}
	We will limit our discussion to the scenario where $m \geq n+d$, and matrix $C^R_c$ has full row rank. The RBLSE problem can then be stated as follows:
	\begin{equation}\label{eq6.1}
		\min_{X \in \R^{n \times d}}\norm{AX-B}_F \quad \textrm{subject to} \quad  CX=D. 
	\end{equation} 
	Consider a real LSE problem 
	\begin{equation}\label{eq6.2}
		\min_{X \in \R^{n \times d}}\norm{A^R_cX-B^R_c}_F \quad \textrm{subject to} \quad  C^R_cX=D^R_c.
	\end{equation}
	\begin{theorem}\label{thm3.3}
		Consider the RBLSE problem outlined in \eqref{eq6.1} and the real LSE problem in \eqref{eq6.2}. Let the QR factorization of \( \left(C^R_c\right)^T \) be given by:
			\begin{equation}\label{eq6.3}
				\left(C^R_c\right)^T = \bar{Q}\begin{bmatrix}
					\bar{R} \\
					0             
				\end{bmatrix},
			\end{equation}
			where \( \bar{Q} \in \mathbb{R}^{n \times n} \) is an orthonormal matrix and \( \bar{R} \in \mathbb{R}^{4p \times 4p} \) is a nonsingular upper triangular matrix. Further, partition the product \( A^R_c \bar{Q} \) as:
			\begin{equation}\label{eq6.4}
				A^R_c \bar{Q} = \left[\bar{P}_1, \bar{P}_2\right],
			\end{equation}
			where \( \bar{P}_1 \in \mathbb{R}^{4m \times 4p} \) and \( \bar{P}_2 \in \mathbb{R}^{4m \times (n - 4p)} \).
		
		Let \( X \in \mathbb{R}^{n \times d} \). Then, $X$ is a real solution of the RBLSE problem \eqref{eq6.1} if and only if \( X \) solves the real LSE problem \eqref{eq6.2}. In this scenario, the unique solution \( X \) with minimum norm can be expressed as:
		\begin{equation}\label{eq6.5}
			X = \bar{Q} \begin{bmatrix}
				\left(\bar{R}^T\right)^{-1} D^R_c \\ 
				\bar{P}_2^\dagger \left(B^R_c - \bar{P}_1 \left(\bar{R}^T\right)^{-1} D^R_c\right)
			\end{bmatrix}.
		\end{equation}
	\end{theorem}
	\begin{proof}
		If $X \in \R^{n \times d}$ is a solution of the real LSE problem \eqref{eq6.2}, then 
		\begin{equation}\label{p11}
			\norm{A^R_cX-B^R_c}_F= \textrm{min},  \quad C^R_cX=D^R_c.
		\end{equation}
		The Frobenius norm of a real matrix remains invariant under orthogonal transformations. Since the matrices $Q_m$, $R_m$, and $S_m$ in \eqref{Qm} are orthogonal, it follows that:
		\begin{equation*}
			\norm{A^R_cX-B^R_c}_F = \norm{Q_m\left(A^R_cX-B^R_c\right)}_F = \norm{R_m\left(A^R_cX-B^R_c\right)}_F =\norm{S_m\left(A^R_cX-B^R_c\right)}_F .
		\end{equation*}
		Using equations \eqref{eq2.2}, \eqref{eq2.3}, \eqref{ec}, along with Lemmas \ref{lem2} and \ref{lem2.3}, we obtain
		\begin{eqnarray*}
			\norm{A^R_cX-B^R_c}_F &=& \frac{1}{2}\norm{\left[\left(A^R_cX-B^R_c\right), Q_m\left(A^R_cX-B^R_c\right), R_m\left(A^R_cX-B^R_c\right), S_m\left(A^R_cX-B^R_c\right)\right]}_F \\
			&=& \frac{1}{2} \norm{\left[A^R_cX, Q_mA^R_cX, R_mA^R_cX, S_mA^R_cX\right]-\left[B^R_c, Q_mB^R_c, R_mB^R_c, S_m B^R_c\right]}_F \\
			&=& \frac{1}{2} \norm{\left[A^R_c, Q_mA^R_c, R_mA^R_c, S_m A^R_c\right]\begin{bmatrix}
					X & 0 & 0 & 0\\
					0 & X & 0 & 0\\
					0 & 0 & X & 0\\
					0 & 0 & 0 & X
				\end{bmatrix}-\left[B^R_c, Q_mB^R_c, R_mB^R_c, S_mB^R_c\right]}_F \\
			&=&  \frac{1}{2} \norm{A^R X^R-B^R}_F \\
			&=&  \frac{1}{2} \norm{\left(A X-B\right)^R}_F \\
			&=&  \norm{A X-B}_F.
		\end{eqnarray*}
		From \eqref{p11}, we obtain
		\begin{equation}\label{p12}
			\norm{A^R_cX-B^R_c}_F  = \norm{A X-B}_F  = \textrm{min},
		\end{equation}
		and
		\begin{equation*}
			\left[C^R_c, Q_p C^R_c, R_p C^R_c, S_p C^R_c\right]	\begin{bmatrix}
				X & 0 & 0 & 0\\
				0 & X & 0 & 0\\
				0 & 0 & X & 0\\
				0 & 0 & 0 & X
			\end{bmatrix} = \left[D^R_c, Q_p D^R_c, R_p D^R_c, S_p D^R_c\right].
		\end{equation*}
		Using \eqref{eq2.3}, we know that $C^R = \left[C^R_c, Q_p C^R_c, R_p C^R_c, S_p C^R_c\right]$ and $D^R=\left[D^R_c, Q_p D^R_c, R_p D^R_c, S_p D^R_c\right]$. Applying this, we get
		\begin{eqnarray}
			C^R X^R&=& D^R, \nonumber\\
			(CX)^R&=& D^R, \nonumber\\
			CX&=& D. \label{p13}
		\end{eqnarray}
		By combining \eqref{p12} and \eqref{p13}, we conclude that 
		$X  \in  \R^{n \times d}$ is a real solution to the RBLSE problem \eqref{eq6.1}, and vice versa. 
		
		To determine the expression for \( X \), we solve the real LSE problem given in \eqref{eq6.2}. Using Lemma \ref{thm2.4}, we derive the explicit expression for \( X \).
	\end{proof}
	Next, we aim to examine how perturbations in $A$, $B$, $C$, and $D$ affect the real solution $X_{RL}$ of the RBLSE problem \eqref{eq6.1}. Let 
	\begin{equation}\label{rdeta1}
		\widehat{A} = A+\Delta A, \quad \widehat{B}=B+\Delta B, \quad \widehat{C}=C+\Delta C, \quad \textrm{and} \quad \widehat{D}=D+\Delta D,
	\end{equation}	
	where $\Delta A$, $\Delta B$, $\Delta C$, and $\Delta D$ represent the perturbations of the input data $A$, $B$, $C$, and $D$, respectively. We assume that the perturbations $\Delta A$, $\Delta B$, $\Delta C$, and $\Delta D$ are small enough to ensure that the perturbed matrix $\widehat{C}^R_c$ retains full row rank. These perturbations are measured normwise by the smallest $\epsilon$ for which
	\begin{equation}\label{deta2}
		\norm{\Delta A}_F \leq \epsilon \norm{A}_F, \quad \norm{\Delta B}_F \leq \epsilon \norm{B}_F, \quad \norm{\Delta C}_F \leq \epsilon \norm{C}_F, \quad \textrm{and} \quad \norm{\Delta D}_F \leq \epsilon \norm{D}_F.
	\end{equation}
	Let $\widehat{X}_{RL}$ be the real solution to the perturbed RBLSE problem 
	\begin{equation}\label{perr1}
		\min_{X \in \R^{n \times d}}\norm{\widehat{A}X-\widehat{B}}_F \quad  \textrm{subject to} \quad \widehat{C}X=\widehat{D},
	\end{equation} 
	and let $\Delta X_{RL}= \widehat{X}_{RL}-X_{RL}$. 
	\begin{theorem}\label{thm4.8}
		Consider the RBLSE problem outlined in \eqref{eq6.1} and the perturbed RBLSE problem described in \eqref{perr1}. If the perturbations $\Delta A$, $\Delta B$, $\Delta C$, and $\Delta D$ are sufficiently small, as described in \eqref{deta2}, then we have
		\begin{equation} \label{url}
			\begin{split}
				\frac{\norm{\Delta X_{RL}}_F}{\norm{X_{RL}}_F} \leq \epsilon &\left(\mathcal{K}_A^R\left(\frac{\norm{D^R_c}_F}{\norm{C^R_c}_F\norm{X_{RL}}_F}+1\right) + \mathcal{K}_{B}^R\left(\frac{\norm{B^R_c}_F}{\norm{A^R_c}_F \norm{X_{RL}}_F}+1\right)\right. \\
				&\left. +\left(\mathcal{K}_B^R\right)^2\left(\frac{\norm{C^R_c}_F}{\norm{A^R_c}_F}\norm{A^R_c\mathcal{L}_r}_2+1\right)\frac{\norm{R_r}_F}{\norm{A^R_c}_F\norm{X_{RL}}_F}\right)+O(\epsilon^2) \equiv U_{RL},
			\end{split}
		\end{equation}
		where
		\begin{eqnarray*}
			\mathcal{K}_B^R&=&\norm{A^R_c}_F \norm{\left(A^R_cP_r\right)^{\dagger}}_2, \; \; \; \mathcal{K}_A^R=\norm{C^R_c}_F \norm{\mathcal{L}_r}_2, \; \; \;  \mathcal{L}_r =\left(I_n-\left(A^R_cP_r\right)^{\dagger}A^R_c\right)\left(C^R_c\right)^{\dagger}, \; \; \; \;  \\     P_r&=&I_n-\left(C^R_c\right)^{\dagger}C^R_c, \; \; \; \; \; \; \; \; \; \; \; R_r=B^R_c-A^R_cX_{RL}.
		\end{eqnarray*}
	\end{theorem}
	\begin{proof}
		The perturbed real LSE problem corresponding to the perturbed RBLSE problem \eqref{perr1} is given by:
		\begin{equation}\label{per4}
			\min_{X \in \R^{n \times d}}\norm{\left(\widehat{A}\right)^R_cX-\left(\widehat{B}\right)^R_c}_F \quad \textrm{subject to} \quad \left(\widehat{C}\right)^R_cX=\left(\widehat{D}\right)^R_c.
		\end{equation}
		Using Theorem \ref{thm3.3}, we know that $\widehat{X}_{RL}$ is the solution to the perturbed real LSE problem \eqref{per4}. From \eqref{rdeta1} and utilizing Lemma \ref{lem2.3}, we have
		$$\left(\widehat{A}\right)^R_c = A^R_c+\left(\Delta A\right)^R_c, \quad \left(\widehat{B}\right)^R_c = B^R_c+\left(\Delta B\right)^R_c, \quad \left(\widehat{C}\right)^R_c = C^R_c+\left(\Delta C\right)^R_c, \quad \left(\widehat{D}\right)^R_c = D^R_c+\left(\Delta D\right)^R_c.$$  
		Thus, the perturbed real LSE problem \eqref{per4} can be rewritten as:
		\begin{equation}\label{prper3}
			\min_{X \in \R^{n \times d}}\norm{\left(A^R_c+\left(\Delta A\right)^R_c\right)X-\left(B^R_c+\left(\Delta B\right)^R_c\right)}_F \quad  \textrm{subject to} \quad \left(C^R_c+\left(\Delta C\right)^R_c\right)X=\left(D^R_c+\left(\Delta D\right)^R_c\right).
		\end{equation}
		Using \eqref{ec} and \eqref{deta2}, we can establish the following bounds for the perturbation:
		\begin{equation}\label{prper4}
			\norm{\left(\Delta A\right)^R_c}_F \leq \epsilon \norm{A^R_c}_F, \quad \norm{\left(\Delta B\right)^R_c}_F \leq \epsilon \norm{B^R_c}_F, \quad \norm{\left(\Delta C\right)^R_c}_F \leq \epsilon \norm{C^R_c}_F, \quad \norm{\left(\Delta D\right)^R_c}_F \leq \epsilon \norm{D^R_c}_F.
		\end{equation}
		With the perturbed problem \eqref{prper3} and the bounds in \eqref{prper4}, and using Theorem \ref{thm3.3}, the sensitivity analysis of the real solution to the RBLSE problem \eqref{eq6.1} reduces to evaluating the sensitivity of the solution to the real LSE problem \eqref{eq6.2}. Consequently, the upper bound $U_{RL}$ for the relative forward error of the real solution to the RBLSE problem can be obtained from \cite[Equation $4.11$]{MR1682424}.
	\end{proof}

	\section{An algebraic technique for complex solution of RBLSE problem}\label{sec3}
	This section focuses on an algebraic approach to derive the complex solution for the RBLSE problem. The method is based on analyzing the solution of the associated complex LSE problem. Suppose
	\begin{eqnarray}
		A&=& M_1+M_2\j \in  \QR^{m \times n}, \quad B= N_1+N_2\j \in \QR^{m \times d}, \label{eq4.1} \\ 
		C&=& R_1+R_2\j \in  \QR^{p \times n}, \quad D= S_1+S_2\j \in \QR^{p \times d}. \label{eq4.2}
	\end{eqnarray}
	We will limit our discussion to the scenario where $m \geq n+d$, and matrix $C^C_c$ has full row rank. The RBLSE problem can then be stated as follows:
	\begin{equation}\label{eq5.1}
		\min_{X \in \C^{n \times d}}\norm{AX-B}_F \quad \textrm{subject to} \quad  CX=D. 
	\end{equation} 
	Consider a complex LSE problem 
	\begin{equation}\label{eq5.2}
		\min_{X \in \C^{n \times d}}\norm{A^C_cX-B^C_c}_F \quad \textrm{subject to} \quad  C^C_cX=D^C_c.
	\end{equation}
	
	\begin{theorem}\label{thm4.3}
		Consider the RBLSE problem outlined in \eqref{eq5.1} and the complex LSE problem in \eqref{eq5.2}. Let the QR factorization of \( (C^C_c)^H \) be given by:
			\begin{equation}\label{eq5.3}
				\left(C^C_c\right)^H = \widetilde{Q}\begin{bmatrix}
					\widetilde{R} \\
					0             
				\end{bmatrix},
			\end{equation}
			where \( \widetilde{Q} \in \mathbb{C}^{n \times n} \) is a unitary matrix and \( \widetilde{R} \in \mathbb{C}^{2p \times 2p} \) is a nonsingular upper triangular matrix. Further, partition the product \( A^C_c \widetilde{Q} \) as:
			\begin{equation}\label{eq5.4}
				A^C_c \widetilde{Q} = \left[\widetilde{P}_1, \widetilde{P}_2\right],
			\end{equation}
			where \( \widetilde{P}_1 \in \mathbb{C}^{2m \times 2p} \) and \( \widetilde{P}_2 \in \mathbb{C}^{2m \times (n - 2p)} \).
		
		Let \( X \in \mathbb{C}^{n \times d} \). Then, $X$ is a complex solution of the RBLSE problem \eqref{eq5.1} if and only if \( X \) solves the complex LSE problem \eqref{eq5.2}. In this scenario, the unique solution \( X \) with minimum norm can be expressed as:
		\begin{equation}\label{eq5.5}
			X = \widetilde{Q} \begin{bmatrix}
				(\widetilde{R}^H)^{-1} D^C_c \\ 
				\widetilde{P}_2^\dagger \big( B^C_c - \widetilde{P}_1 (\widetilde{R}^H)^{-1} D^C_c \big)
			\end{bmatrix}. 
		\end{equation}
	\end{theorem}
	\begin{proof}
		If $X \in \C^{n \times d}$ is a solution of the complex LSE problem \eqref{eq5.2}, then 
		\begin{equation}\label{p1}
			\norm{A^C_cX-B^C_c}_F= \textrm{min},  \quad  C^C_cX=D^C_c.
		\end{equation}
		The Frobenius norm of a complex matrix remains invariant under unitary transformations. Since the matrix $P_m$ in \eqref{Pm} is unitary, it follows that:
		\begin{equation*}
			\norm{A^C_cX-B^C_c}_F = \norm{P_m\left(A^C_cX-B^C_c\right)}_F .
		\end{equation*}
		Using equations \eqref{eq2.2}, \eqref{eq2.4}, \eqref{er}, along with Lemmas \ref{lem22} and \ref{lem2.3}, we obtain
		\begin{eqnarray*}
			\norm{A^C_cX-B^C_c}_F &=& \frac{1}{\sqrt{2}}\norm{\left[\left(A^C_cX-B^C_c\right), P_m\left(A^C_cX-B^C_c\right)\right]}_F \\
			&=& \frac{1}{\sqrt{2}} \norm{\left[A^C_cX, P_mA^C_cX\right]-\left[B^C_c, P_mB^C_c\right]}_F \\
			&=& \frac{1}{\sqrt{2}} \norm{\left[A^C_c, P_mA^C_c\right]\begin{bmatrix}
					X & 0 \\
					0 & X
				\end{bmatrix}-\left[B^C_c, P_mB^C_c\right]}_F \\
			&=&  \frac{1}{\sqrt{2}} \norm{A^C X^C-B^C}_F \\
			&=&  \frac{1}{\sqrt{2}} \norm{\left(A X-B\right)^C}_F \\
			&=&  \norm{A X-B}_F.
		\end{eqnarray*}
		From \eqref{p1}, we obtain
		\begin{equation}\label{p2}
			\norm{A^C_cX-B^C_c}_F  = \norm{A X-B}_F  = \textrm{min},
		\end{equation}
		and
		\begin{equation*}
			\left[C^C_c, P_p C^C_c\right]	\begin{bmatrix}
				X & 0  \\
				0 & X 
			\end{bmatrix} = \left[D^C_c, P_p D^C_c\right].
		\end{equation*}
		Using \eqref{eq2.4}, we know that $C^C = \left[C^C_c, P_p C^C_c\right]$ and $D^C=\left[D^C_c, P_p D^C_c\right]$. Applying this, we get
		\begin{eqnarray}
			C^C X^C&=& D^C, \nonumber\\
			\left(CX\right)^C&=& D^C, \nonumber\\
			CX&=& D. \label{p3}
		\end{eqnarray}
		By combining \eqref{p2} and \eqref{p3}, we conclude that 
		$X  \in  \C^{n \times d}$ is a complex solution to the RBLSE problem \eqref{eq5.1}, and vice versa.
		
		To find the expression for $X$, we solve the complex LSE problem \eqref{eq5.2}. Using Lemma \ref{thm2.4}, we derive the explicit expression for \( X \).
		\end{proof}
	
	Next, we aim to examine how perturbations in $A$, $B$, $C$, and $D$ affect the complex solution $X_{CL}$ of the RBLSE problem  \eqref{eq5.1}. Let 
	\begin{equation}\label{rdeta2}
		\widehat{A} = A+\Delta A, \quad \widehat{B}=B+\Delta B, \quad \widehat{C}=C+\Delta C, \quad \textrm{and} \quad \widehat{D}=D+\Delta D,
	\end{equation}	
	where $\Delta A$, $\Delta B$, $\Delta C$, and $\Delta D$ represent the perturbations of the input data $A$, $B$, $C$, and $D$, respectively. We assume that the perturbations $\Delta A$, $\Delta B$, $\Delta C$, and $\Delta D$ are small enough to guarantee that the perturbed matrix $\widehat{C}^C_c$ retains full row rank. These perturbations are measured normwise by the smallest $\epsilon$ for which
	\begin{equation}\label{deta1}
		\norm{\Delta A}_F \leq \epsilon \norm{A}_F, \quad \norm{\Delta B}_F \leq \epsilon \norm{B}_F, \quad \norm{\Delta C}_F \leq \epsilon \norm{C}_F, \quad \textrm{and} \quad \norm{\Delta D}_F \leq \epsilon \norm{D}_F.
	\end{equation}
	Let $\widehat{X}_{CL}$ be the complex solution to the perturbed RBLSE problem 
	\begin{equation}\label{per1}
		\min_{X \in \C^{n \times d}}\norm{\widehat{A}X-\widehat{B}}_F \quad  \textrm{subject to} \quad \widehat{C}X=\widehat{D},
	\end{equation}
	and let $\Delta X_{CL}= \widehat{X}_{CL}-X_{CL}$. 
	\begin{theorem}\label{thm4.7}
		Consider the RBLSE problem outlined in \eqref{eq5.1} and the perturbed RBLSE problem described in \eqref{per1}. If the perturbations $\Delta A$, $\Delta B$, $\Delta C$, and $\Delta D$ are sufficiently small, as described in \eqref{deta1}, then we have
		\begin{equation}\label{ucl}
			\begin{split}
				\frac{\norm{\Delta X_{CL}}_F}{\norm{X_{CL}}_F} \leq \epsilon &\left(\mathcal{K}_A^C\left(\frac{\norm{D^C_c}_F}{\norm{C^C_c}_F\norm{X_{CL}}_F}+1\right) + \mathcal{K}_{B}^C\left(\frac{\norm{B^C_c}_F}{\norm{A^C_c}_F \norm{X_{CL}}_F}+1\right)\right. \\
				&\left. +\left(\mathcal{K}_B^C\right)^2\left(\frac{\norm{C^C_c}_F}{\norm{A^C_c}_F}\norm{A^C_c\mathcal{L}_c}_2+1\right)\frac{\norm{R_c}_F}{\norm{A^C_c}_F\norm{X_{CL}}_F}\right)+O(\epsilon^2) \equiv U_{CL},
			\end{split}
		\end{equation}
		where
		\begin{eqnarray*}
			\mathcal{K}_B^C&=&\norm{A^C_c}_F \norm{\left(A^C_cP_c\right)^{\dagger}}_2, \; \; \; \mathcal{K}_A^C=\norm{C^C_c}_F \norm{\mathcal{L}_c}_2, \; \; \;  \mathcal{L}_c =\left(I_n-\left(A^C_cP_c\right)^{\dagger}A^C_c\right)\left(C^C_c\right)^{\dagger}, \; \; \; \;  \\     P_c&=&I_n-\left(C^C_c\right)^{\dagger}C^C_c, \; \; \; \; \; \; \; \; \; \; \; R_c=B^C_c-A^C_cX_{CL}.
		\end{eqnarray*}
	\end{theorem}
	\begin{proof}
		The perturbed complex LSE problem corresponding to the perturbed RBLSE problem \eqref{per1} is given by:
		\begin{equation}\label{per2}
			\min_{X \in \C^{n \times d}}\norm{\left(\widehat{A}\right)^C_cX-\left(\widehat{B}\right)^C_c}_F \quad \textrm{subject to} \quad \left(\widehat{C}\right)^C_cX=\left(\widehat{D}\right)^C_c.
		\end{equation}
		Using Theorem \ref{thm4.3}, we know that $\widehat{X}_{CL}$ is the solution to the perturbed complex LSE problem \eqref{per2}. From \eqref{rdeta2} and utilizing Lemma \ref{lem2.3}, we have
		$$\left(\widehat{A}\right)^C_c = A^C_c+\left(\Delta A\right)^C_c, \quad \left(\widehat{B}\right)^C_c = B^C_c+\left(\Delta B\right)^C_c, \quad \left(\widehat{C}\right)^C_c = C^C_c+\left(\Delta C\right)^C_c, \quad \left(\widehat{D}\right)^C_c = D^C_c+\left(\Delta D\right)^C_c.$$ 
		Thus, the perturbed complex LSE problem \eqref{per2} can be rewritten as:
		\begin{equation}\label{prper1}
			\min_{X \in \C^{n \times d}}\norm{\left(A^C_c+\left(\Delta A\right)^C_c\right)X-\left(B^C_c+\left(\Delta B\right)^C_c\right)}_F \quad  \textrm{subject to} \quad \left(C^C_c+\left(\Delta C\right)^C_c\right)X=\left(D^C_c+\left(\Delta D\right)^C_c\right).
		\end{equation}
		Using \eqref{er} and \eqref{deta1}, we can establish the following bounds for the perturbation:
		\begin{equation}\label{prper2}
			\norm{\left(\Delta A\right)^C_c}_F \leq \epsilon \norm{A^C_c}_F, \quad \norm{\left(\Delta B\right)^C_c}_F \leq \epsilon \norm{B^C_c}_F, \quad \norm{\left(\Delta C\right)^C_c}_F \leq \epsilon \norm{C^C_c}_F, \quad \norm{\left(\Delta D\right)^C_c}_F \leq \epsilon \norm{D^C_c}_F.
		\end{equation}
		With the perturbed problem \eqref{prper1} and the bounds in \eqref{prper2}, and using Theorem \ref{thm4.3}, the sensitivity analysis of the complex solution to the RBLSE problem \eqref{eq5.1} reduces to evaluating the sensitivity of the solution to the complex LSE problem \eqref{eq5.2}. Consequently, the upper bound $U_{CL}$ for the relative forward error of the complex solution to the RBLSE problem can be obtained from \cite[Equation $4.11$]{MR1682424}.
	\end{proof}

	\section{Numerical Verification}\label{sec5}
 In this section, we present numerical algorithms for computing two special solutions (real and complex) of the RBLSE problem. Each algorithm is complemented with a detailed computational complexity analysis through a step-by-step flop count. Additionally, we validate these algorithms using numerical examples to assess their performance and reliability.
		\begin{algorithm} [H]
		\caption{For the real solution of the RBLSE problem \eqref{eq6.1}}\label{alg2}
		\vspace{0.2cm}
		\textbf{Input:} $A= A_0+ A_1\i+ A_2\j+ A_3\k \in \QR^{m \times n}$, $B= B_0+ B_1\i+ B_2\j+ B_3\k \in \QR^{m \times d}$, $C= C_0+ C_1\i+ C_2\j+ C_3\k \in  \QR^{p \times n}$, and $D= D_0+ D_1\i+ D_2\j+ D_3\k \in \QR^{p \times d}.$ We assume $m \geq n+d$ and that matrix $C^R_c$ has full row rank.\\
		\textbf{Output:} $X \in \R^{n \times d}$.
		\begin{enumerate}[noitemsep]
			\item Find the QR factorization of $\left(C^R_c\right)^T$, as described in \eqref{eq6.3}.
			\item Partition the matrix $A^R_c\bar{Q}$ as shown in \eqref{eq6.4}. 
			\item The real solution $X$ to the RBLSE problem \eqref{eq6.1} is given by \eqref{eq6.5}. 
		\end{enumerate}
		\vspace{-0.2cm}
	\end{algorithm}
	\begin{algorithm} [H]
		\caption{For the complex solution of the RBLSE problem \eqref{eq5.1}}\label{alg1}
		\vspace{0.2cm}
		\textbf{Input:} $A= M_1+ M_2\j \in \QR^{m \times n}$, $B= N_1+ N_2\j \in \QR^{m \times d}$, $C= R_1+ R_2\j \in  \QR^{p \times n}$, and $D= S_1+ S_2\j \in \QR^{p \times d}.$ We assume $m \geq n+d$ and that matrix $C^C_c$ has full row rank.\\
		\textbf{Output:} $X \in \C^{n \times d}$.
		\begin{enumerate}[noitemsep]
			\item Find the QR factorization of $\left(C^C_c\right)^H$, as described in \eqref{eq5.3}.
			\item Partition the matrix $A^C_c\widetilde{Q}$ as shown in \eqref{eq5.4}. 
			\item The complex solution $X$ to the RBLSE problem \eqref{eq5.1} is given by \eqref{eq5.5}.
		\end{enumerate}
		\vspace{-0.2cm}
	\end{algorithm}

		Next, to assess the performance of the algorithms, a step-by-step flop count is presented. This detailed analysis highlights the computational cost at each stage, providing insights into the efficiency of the algorithm.
		First, to analyze the efficiency of the real solution algorithm, a detailed step-by-step flop count is provided below.
	\begin{table}[H]
			\centering
			\renewcommand{\arraystretch}{1.5} 
			\begin{tabular}{@{}p{0.15\linewidth}@{\hspace{0.02\linewidth}}p{0.35\linewidth}@{\hspace{0.01\linewidth}}p{0.45\linewidth}@{}}
				\toprule
				\textbf{Step} & \textbf{Description} & \textbf{Flop Count} \\
				\midrule
				1 & QR decomposition of $\left(C^R_c\right)^T$ & $O(32np^2)$ \\
				
				2 & Compute $A^R_c \bar{Q}$ & $O(8m n^2 - 4mn)$ \\
				
				3 & Partition $A^R_c \bar{Q}$ into $\bar{P}_1$ and $\bar{P}_2$  & $O(1)$ \\
				
				4 & Compute $\bar{P}_2^\dagger$ &  
				$O\big(24m(n - 4p)^2  + 10 (n - 4p)^3\big)$ \\
				
				5 & Solve $\left(\bar{R}^T\right)^{-1}D^R_c$  & $O(16p^2d)$ \\
				
				6 & Compute $\bar{P}_1 \left(\bar{R}^T\right)^{-1}D^R_c$ & $O(32mpd-4md)$ \\
				
				7 & Compute $B^R_c - \bar{P}_1 \left(\bar{R}^T\right)^{-1}D^R_c$  & $O(4md)$ \\
				
				8 & Compute $\bar{P}_2^\dagger (B^R_c - \bar{P}_1 \left(\bar{R}^T\right)^{-1}D^R_c)$  & $O(8m(n - 4p)d)$ \\
				
				9 & Compute $X $  & $O(2n^2d)$ \\
				
				\midrule
				& \textbf{Total Flop Count}   & 
				$O\big(
				32np^2 + 8mn^2 - 4mn+ 24m(n - 4p)^2  + 10 (n - 4p)^3 
				+ 16p^2d + 32mpd-4md + 4md + 8m(n - 4p)d + 2n^2d
				\big)$ \\
				\bottomrule
			\end{tabular}
			\caption{Flop count for the computational steps to find the real solution of the RBLSE problem.}
		\label{tab:flop_count}
	\end{table}
	To analyze the efficiency of the complex solution algorithm, a detailed step-by-step flop count is provided below. This analysis emphasizes the additional computations required due to the complex structure.
	\begin{table}[H]
			\centering
			\renewcommand{\arraystretch}{1.5} 
			\begin{tabular}{@{}p{0.15\linewidth}@{\hspace{0.02\linewidth}}p{0.35\linewidth}@{\hspace{0.01\linewidth}}p{0.45\linewidth}@{}}
				\toprule
				\textbf{Step} & \textbf{Description} & \textbf{Flop Count} \\
				\midrule
				1 & QR decomposition of $\left(C^C_c\right)^H$ & $O(32np^2)$ \\
				
				2 & Compute $A^C_c \widetilde{Q}$ & $O(16mn^2-4mn)$ \\
				
				3 & Partition $A^C_c \widetilde{Q}$ into $\widetilde{P}_1$ and $\widetilde{P}_2$ & $O(1)$ \\
				
				4 & Compute $\widetilde{P}_2^\dagger$ &  
				$O\big(48m(n - 2p)^2  +40 (n - 2p)^3\big)$ \\
				
				5 & Solve $\left(\widetilde{R}^H\right)^{-1} D^C_c$ & $O(16p^2d+10pd)$ \\
				
				6 & Compute $\widetilde{P}_1 \left(\widetilde{R}^H\right)^{-1} D^C_c$ & $O(32mpd-4md)$ \\
				
				7 & Compute $B^C_c - \widetilde{P}_1 \left(\widetilde{R}^H\right)^{-1} D^C_c$ & $O(4md)$ \\
				
				8 & Compute $\widetilde{P}_2^\dagger (B^C_c - \widetilde{P}_1 \left(\widetilde{R}^H\right)^{-1} D^C_c)$ & $O(16m(n - 2p)d)$ \\
				
				9 & Compute $X $ & $O(8n^2d)$ \\
				
				\midrule
				& \textbf{Total Flop Count} & 
				$O\big(
				32np^2 + 16mn^2-4mn + 48m(n - 2p)^2  + 40 (n - 2p)^3 
				+ 16p^2d +10pd+ 32mpd-4md + 4md + 16m(n - 2p)d + 8n^2d
				\big)$ \\
				\bottomrule
			\end{tabular}
			\caption{Flop count for the computational steps to find the complex solution of the RBLSE problem.}
		\label{tab:flop_count_complex}
	\end{table}

	Next, we provide examples to evaluate the effectiveness of the proposed algorithms. All computations were performed using a computer equipped with an Intel Core $i9-12900K$ processor at $3200$ MHz and $32$ GB of RAM, running MATLAB $R2024b$ software.
	
		The following example validates the efficiency of the algorithms by comparing the real and complex solutions in terms of computational accuracy and CPU runtime.
		\begin{exam} \label{ex1}
			To evaluate the accuracy and performance of the proposed algorithms for solving the RBLSE problem, we use random data matrices of varying sizes. The problem matrices are defined as follows:
			\begin{eqnarray*}
				A &=& A_0 + A_1\i + A_2\j + A_3\k = M_1 + M_2\j \in \QR^{m \times n}, \quad
				B = B_0 + B_1\i + B_2\j + B_3\k = N_1 + N_2\j \in \QR^{m \times d}, \\
				C &=& C_0 + C_1\i + C_2\j + C_3\k = R_1 + R_2\j \in \QR^{p \times n}, \quad
				D = D_0 + D_1\i + D_2\j + D_3\k = S_1 + S_2\j \in \QR^{p \times d}. 
			\end{eqnarray*}
			The real matrix components \( A_i, B_i, C_i, \) and \( D_i \) are generated using random values:
			\begin{align*}
				A_i &= rand(m, n) \in \R^{m \times n}, \quad i = 0,1,2,3, \\
				B_i &= rand(m, d) \in \R^{m \times d}, \quad i = 0,1,2,3, \\
				C_i &= rand(p, n) \in \R^{p \times n}, \quad i = 0,1,2,3, \\
				D_i &= rand(p, d) \in \R^{p \times d}, \quad i = 0,1,2,3.
			\end{align*}
			The matrix dimensions are determined by \( m = 30t \), \( n = 10t \), \( p = 2t \), and \( d = 2 \), where \( t \) is a positive integer. We apply Algorithms \ref{alg2} and \ref{alg1} to compute the real and complex solutions of the RBLSE problem, respectively.
			\begin{enumerate}
				\item[(a)] Let \( X_{RL} \) be the real solution of the RBLSE problem:
				\[
				\min_{X, R_r} \|R_r\|_F \quad \text{subject to} \quad A X = B + R_r, \quad C X = D.
				\]
				The accuracy of \( X_{RL} \) is measured using the following errors:
				\[
				\epsilon_1 = \log_{10}(\|A X_{RL} - (B + R_r)\|_F), \quad \epsilon_2 = \log_{10}(\|C X_{RL} - D\|_F).
				\]
				
				\item[(b)] Let \( X_{CL} \) be the complex solution of the RBLSE problem:
				\[
				\min_{X, R_c} \|R_c\|_F \quad \text{subject to} \quad A X = B + R_c, \quad C X = D.
				\]
				The accuracy of \( X_{CL} \) is measured using:
				\[
				\epsilon_3 = \log_{10}(\|A X_{CL} - (B + R_c)\|_F), \quad \epsilon_4 = \log_{10}(\|C X_{CL} - D\|_F).
				\]
				\end{enumerate}
			To thoroughly evaluate the computational performance, the CPU times for each solution are recorded and averaged over $50$ trials to minimize variability. The average CPU times (in seconds) for computing \( X_{RL} \) and \( X_{CL} \) are denoted by \( t_r \) and \( t_c \), respectively.
		\end{exam}
		\begin{table}[H]
				\centering
				\begin{tabular}{c c c c c } 
					\toprule
					\( t \) & \( \epsilon_1 \) & \( \epsilon_2 \) & \( \epsilon_3 \) & \( \epsilon_4 \)  \\
					\midrule
					\( 1 \)  & \( -15.1057 \) & \( -15.0513 \) & \( -15.5746 \) & \( -15.2797 \) \\
					\( 3 \)  & \( -15.1474 \) & \( -14.5131 \) & \( -15.4137 \) & \( -14.7818 \) \\
					\( 5 \) & \( -14.9570 \) & \( -14.4140 \) & \( -15.3090 \) & \( -14.4512 \)  \\
					\( 7 \) & \( -14.8277 \) & \( -14.2473 \) & \( -15.3418 \) & \( -14.5228 \)  \\
					\( 9 \)  & \( -14.8800 \) & \( -14.0209 \) & \( -15.3164 \) & \( -14.2537 \) \\
					\bottomrule
				\end{tabular}
				\caption{Accuracy of Algorithms \ref{alg2} and \ref{alg1} in computing the real and complex solutions of the RBLSE problem.}
				\label{tab12}
		\end{table}
		
		\noindent Table \ref{tab12} demonstrates that the two algorithms achieve high accuracy, with the computational errors consistently below \( -14 \) for all tested values of \( t \). This highlights the accuracy of the proposed methods.
		
		\begin{figure}[H]
				\centering
				\includegraphics[width=0.8\textwidth]{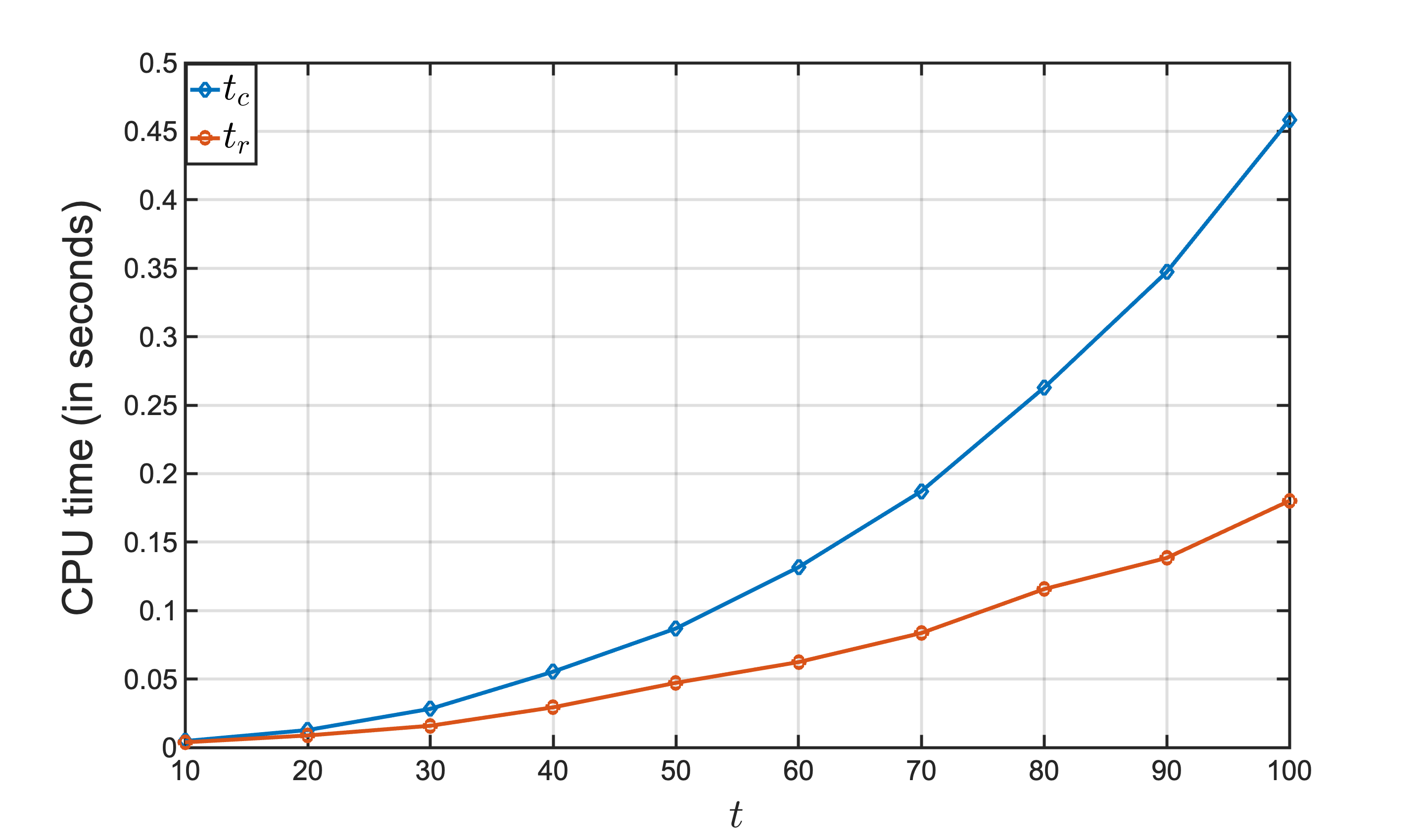} 
				\caption{Comparison of CPU time for computing real and complex solution of RBLSE problem}
				\label{fig1}
		\end{figure}
		
		\noindent Figure \ref{fig1} compares the average CPU times for the two algorithms. The results show that the algorithm for computing real solution takes less time compared to the algorithm for computing the complex solution of the RBLSE problem. This is consistent with the theoretical complexity of the algorithms.
		
		The following example compares the theoretical upper bounds with the actual relative forward errors for the real and complex solutions of the RBLSE problem, demonstrating the reliability of the proposed bounds.
		\begin{exam}\label{ex1.1}
			Building on Example \ref{ex1}, we now introduce random perturbations \( \Delta A \), \( \Delta B \), \( \Delta C \), and \( \Delta D \) to the input matrices \( A \), \( B \), \( C \), and \( D \), respectively. These perturbations are used to examine the sensitivity of the RBLSE problem to small changes in the input data. Specifically, we evaluate how these perturbations affect the computed real solution \( X_{RL} \) and the complex solution \( X_{CL} \). The magnitude of these perturbations is measured normwise by the smallest \( \epsilon \), which ensures the changes are controlled and small.
			
			To quantify the impact of the perturbations, we compute the exact relative forward errors for each solution as follows:
			\[
			\frac{\|\widehat{X}_{RL} - X_{RL}\|_F}{\|X_{RL}\|_F} \quad 	\text{and} \quad 
			\frac{\|\widehat{X}_{CL} - X_{CL}\|_F}{\|X_{CL}\|_F}.
			\]
			Here, \( \widehat{X}_{RL} \) and \( \widehat{X}_{CL} \) denote the computed solutions of the perturbed RBLSE problem, while \( X_{RL} \) and \( X_{CL} \) are the original solutions.
			
			In addition to the exact relative forward errors, we compute the corresponding upper bounds \( U_{RL} \) and \( U_{CL} \) for these errors. These bounds are derived using equations \eqref{url} and \eqref{ucl}, respectively, which provide theoretical estimates of the maximum possible errors. 
			
			Table \ref{tab13} presents the exact relative forward errors and their corresponding upper bounds for various values of \( t \) and \( \epsilon \), allowing us to evaluate the relationship between the perturbation magnitude and the stability of the solutions.
		\end{exam}
		
		\begin{table}[H]
			\begin{center}
						\begin{tabular}{c c c c c c } 
							\toprule
							\( t \) & \( \epsilon \) & \( \frac{\|\widehat{X}_{RL} - X_{RL}\|_F}{\|X_{RL}\|_F} \) & \( U_{RL} \) & \( \frac{\|\widehat{X}_{CL} - X_{CL}\|_F}{\|X_{CL}\|_F} \) & \( U_{CL} \)  \\
							\midrule
							\( 1 \) & \( 1.2957 \times 10^{-13} \)  & \( 4.5138 \times 10^{-13} \) & \( 7.6470 \times 10^{-12} \) & \( 3.8724 \times 10^{-13} \) & \( 5.6680 \times 10^{-12} \)  \\ 
							& \( 1.3279 \times 10^{-10} \)  & \( 4.3088 \times 10^{-10} \) & \( 7.8370 \times 10^{-9} \) & \( 1.8961 \times 10^{-10} \) & \( 5.8088 \times 10^{-9} \)  \\ 
							& \( 1.2770 \times 10^{-7} \)  & \( 3.0154 \times 10^{-7} \)  & \( 7.5367 \times 10^{-6} \) & \( 2.8896 \times 10^{-7} \)  & \( 5.5863 \times 10^{-6} \)  \\ 
							\midrule       
							\( 5 \) & \( 1.2312 \times 10^{-11} \)  & \( 6.1760 \times 10^{-11} \) & \( 4.7674 \times 10^{-9} \) & \( 3.1584 \times 10^{-11} \) & \( 1.7960 \times 10^{-9} \)  \\ 
							& \( 1.5256 \times 10^{-9} \)   & \( 9.3350 \times 10^{-9} \)  & \( 5.9071 \times 10^{-7} \) & \( 4.8613 \times 10^{-9} \)  & \( 2.2254 \times 10^{-7} \)  \\ 
							& \( 1.4801 \times 10^{-8} \)   & \( 7.9283 \times 10^{-8} \)  & \( 5.7313 \times 10^{-6} \) & \( 3.3857 \times 10^{-8} \)  & \( 2.1592 \times 10^{-6} \)  \\ 
							\midrule
							\( 9 \) & \( 1.2645 \times 10^{-12} \)  & \( 6.5592 \times 10^{-12} \) & \( 4.5359 \times 10^{-10} \) & \( 4.0664 \times 10^{-12} \) & \( 2.4556 \times 10^{-10} \)  \\ 
							& \( 1.2593 \times 10^{-10} \)  & \( 6.6261 \times 10^{-10} \) & \( 4.5171 \times 10^{-8} \) & \( 4.0808 \times 10^{-10} \) & \( 2.4454 \times 10^{-8} \)  \\ 
							& \( 1.5340 \times 10^{-7} \)   & \( 7.1997 \times 10^{-7} \)  & \( 5.5023 \times 10^{-5} \) & \( 3.8176 \times 10^{-7} \)  & \( 2.9788 \times 10^{-5} \)  \\ 
							\bottomrule
						\end{tabular}
					\caption{Comparison of relative forward errors and their upper bounds for the real and complex solutions of a perturbed RBLSE problem.} 
					\label{tab13}
			\end{center}
		\end{table}
		
		\noindent Table \ref{tab13} illustrates that the exact relative forward errors for the real and complex solutions of the RBLSE problem are consistently smaller than their respective upper bounds \( U_{RL} \) and \( U_{CL} \), for all considered values of \( t \) and \( \epsilon \). These results validate the accuracy and reliability of the derived upper bounds.
		
		The following example evaluates the accuracy of the proposed algorithms for computing the real and complex solutions of the RBLSE problem.
		\begin{exam} \label{ex4.1}
			We define the problem parameters as follows:
			
			Let \( A \in \QR^{m \times n} \) and \( C \in \QR^{p \times n} \) as defined in Example \ref{ex1}. For \( i = 0, 1\), let \( X_i = rand(n, d) \in \R^{n \times d} \). Using these components, we define:
			\[
			X_R = X_0 \in \R^{n \times d}, \quad 
			X_C = X_0 + X_1\i \in \C^{n \times d}
			\]
			Here, the parameters \( m \), \( p \), \( n \), and \( d \) remain consistent with Example \ref{ex1} and depend on the variable \( t \).
			
			For the real solution, we set \( B = A X_R \) and \( D = C X_R \), which ensures that \( X_R \) is the exact real solution of the RBLSE problem. Algorithm \ref{alg2} is employed to compute the approximate real solution \( \widehat{X}_R \). The accuracy of the computed solution is evaluated using the Frobenius norm error:
			\[
			\epsilon_R = \|X_R - \widehat{X}_R\|_F.
			\]
			
			For the complex solution, we set \( B = A X_C \) and \( D = C X_C \), making \( X_C \) the exact complex solution of the RBLSE problem. Algorithm \ref{alg1} is used to compute the approximate complex solution \( \widehat{X}_C \), and the error is defined as:
			\[
			\epsilon_C = \|X_C - \widehat{X}_C\|_F.
			\]
			
			The variable \( t \) determines the dimensions of the matrices \( A \) and \( C \), and the errors \( \epsilon_R \) and \( \epsilon_C \) are computed for various values of \( t \). The results are summarized in Table \ref{tab16}, which shows the relationship between \( t \) and the computational accuracy of the algorithms.
		\end{exam}
		\begin{table}[H]
		\begin{center}
					\begin{tabular}{ c c c } 
						\toprule
						$t$ & $\epsilon_R = \norm{X_R-\widehat{X}_R}_F$ & $\epsilon_C = \norm{X_C-\widehat{X}_C}_F$  \\
						\midrule
						$1$               & $7.0655 \times 10^{-15} $     & $3.5401 \times 10^{-15} $ \\
						$3$               & $1.3490 \times 10^{-14} $     & $1.0711 \times 10^{-15} $ \\
						$5$               & $2.3275 \times 10^{-14} $     & $2.2491 \times 10^{-14} $ \\
						$7$               & $5.4268 \times 10^{-14} $     & $3.7398 \times 10^{-14} $ \\
						$9$              & $5.8569 \times 10^{-14} $     & $5.2257 \times 10^{-14} $ \\ 
						\bottomrule
					\end{tabular}
					\caption{Computational accuracy of Algorithms \ref{alg2} and \ref{alg1} for computing the real and complex solutions of the RBLSE problem.} \label{tab16}
			\end{center}
		\end{table}
		
		\noindent Table \ref{tab16} demonstrates that the errors \( \epsilon_R \) and \( \epsilon_C \) are consistently below \( 10^{-14} \) for all tested values of \( t \). These results confirm the high accuracy of the proposed algorithms for computing the real and complex solutions of the RBLSE problem, showcasing their effectiveness and precision across varying problem sizes.
	
	\section{Conclusions} \label{sec6}
	We have presented an algebraic method for solving the RBLSE problem by transforming it into real and complex  LSE problems. This transformation leverages the real and complex representations of reduced biquaternion matrices, enabling efficient computation of real and complex solutions to the RBLSE problem. Additionally, we have calculated the upper bound for the relative forward error associated with these solutions, demonstrating the accuracy and reliability of our method in addressing the RBLSE problem. Future research will focus on exploring their applications in image and signal processing.
	\section{Declarations}
	
	\subsection*{Author's Contributions}
	The first author played a crucial role in overseeing the project and providing valuable guidance. Neha Bhadala made significant contributions at various stages of the research, including conceptu-alization, methodology design, software implementation, validation, and formal analysis. Additio-nally, Neha Bhadala took the lead in drafting the original manuscript and actively participated in the subsequent writing and editing phases.
	\subsection*{Conflict of interest}
	This work does not have any conflicts of interest.
	\subsection*{Acknowledgement}
	Neha Bhadala expresses gratitude to the Government of India for funding her research under the Prime Minister’s Research Fellowship program (PMRF ID: $210059$).
	\bibliography{RBLSE_revision}
\end{document}